\documentclass{amsart}
\usepackage{ifpdf}
\usepackage{amsmath, amsthm, amsfonts, amssymb, amscd}
\usepackage{hyperref}
  \usepackage{epsfig} %For pictures: screened artwork should be set up with an 85 or 100 line screen
\newtheorem{thm}{Theorem}

\newtheorem{lemma}[thm]{Lemma}
\newtheorem{proposition}{Proposition}

\theoremstyle{definition}

\newcommand{\R}{\mathbb{R}}

\newcommand{\N}{\mathbb{N}}

\DeclareMathOperator{\supp}{supp}

\title[Maximizing measures supported on a periodic orbit]
      {Density of the set of endomorphisms with a maximizing measure supported on a periodic orbit}

\author[T. C. Batista, J. S. Gonschorowski and F. A. Tal]{T. Batista, J. Gonschorowski and F. Tal}


 \email{tbatista@ime.usp.br}
\email{jsg@ime.usp.br}
 \email{fabiotal@ime.usp.br}

\thanks{The first and second authors were supported by CNPq, the third author was partially supported by CAPES, CNPq, FAPESP}

\begin{document}
\maketitle

%\centerline{\scshape Tatiane C. Batista} \medskip {\footnotesize
% please put the address of the first author  \centerline{Instituto de Matem\'atica e Estat\'istica,}   \centerline{Universidade de S\~ao Paulo. Rua do Mat\~ao 1010}   \centerline{05508-090, S\~ao Paulo SP, Brazil}   } % Do not forget to end the {\footnotesize by the sign }

% Enter the first author's name and address: % \centerline{\scshape Juliano S. Gonschorowski} \medskip {\footnotesize % please put the address of the first author  \centerline{Universidade Tecnologia Federal do Paran\'a - UTFPR,}   \centerline{Rua Presidente Zacarias, 875 - Bloco M}   \centerline{85015-430, Guarapuava-PR, Brazil}    } % Do not forget to end the {\footnotesize by the sign }

%\centerline{\scshape F\'abio A. Tal} \medskip {\footnotesize  % please put the address of the second  and third author  \centerline{ Instituto de Matem\'atica e Estat\'istica, }    \centerline{Universidade de S\~ao Paulo. Rua do Mat\~ao 1010}   \centerline{05508-090, S\~ao Paulo SP, Brazil} } \medskip

% The name of the associate editor will be entered by an editorial staff
% "Communicated by the associate editor name" is not needed for special issue.
% \centerline{(Communicated by the associate editor name)}

%The abstract of your paper
\begin{abstract}
Let $M$ be a  compact $n$-dimensional Riemanian manifold, End($M$) the set of the endomorphisms of $M$ with the usual $\mathcal{C}^0$ topology and $\phi:M\to\R$ continuous. 
We prove, extending the main result of \cite{ATnonlin}, that there exists a dense subset of $\mathcal{A}$ of End($M$) such that, if $f\in\mathcal{A}$, there exists a $f$ invariant measure $\mu_{\max}$ supported on a periodic orbit that maximizes the integral of $\phi$ among all $f$ invariant Borel probability measures.
%\textbf{200} words.
\end{abstract}

%The title of your section 1
\section{Introduction}

A relatively new field of study, {\it ergodic optimization} has displayed under a new point of view several distinct problems in dynamical systems, and enjoyed the benefits of allying techniques from optimization theory and ergodic theory to address  them.  Its usual setup is a dynamical system $f:X\to X$, where $X$ is a topological space, and a potential function $\phi:X\to\R$, and the prototypical problem in the field is to determine, among all $f$ invariant Borel probability measures $\mathcal{M}_{inv}(f)$, if there exists measures that maximize the functional $P_{\phi}:\mathcal{M}_{inv}(f)\to\R,\, P_\phi(\mu)=\int \phi d\mu$ and to further characterize these maximizing measures in term of their support.

Several problems can be put under this context, like finding Lyapunov exponents, action minimizing solutions to Lagrangian systems and the zero temperature limits of Gibbs equilibrium states in thermodynamical formalism. Some of the first ideas of the field appeared in the early work \cite{CG} and a very good introduction to the subject is \cite{jenkinsondcds}, where the fundamental results of the theory are displayed alongside the main lines of research.
One of these research lines seeks to determine, when $X$ is compact and $\phi$ is continuous, what are the typical support of the maximizing measures (note that the existence of at least one maximizing measure is assured in this case by the compactness of the set of invariant probability measures in the weak-* topology). This is inspired by the classical conjecture of Ma\~n\'e that, generically, the measures that minimize the Lagrangian action in Lagrangian flows are supported in periodic orbits. 

There are some different conceptual approaches to this question. First, one may be interested in a specific dynamical property (as, for instance, Lyapunov exponents or rotation numbers) and so the potential is determined by the choice of the dynamical system. Examples of this are \cite{jenkinsonmorris, garibaldinonlin, CDI, CLT}. Another approach, followed for instance in \cite{bouschwalters, Lopesasterisque, Morrisdcds2000}, involves fixing a dynamical system $f$, usually with some specific dynamical condition like hiperbolicity or expansiveness, and varying the potential in a suitable space.

In this work we follow yet a different line, searching to understand how the maximizing measures behave when the potential is fixed, but the dynamics are allowed to change in a given space. In \cite{ATfundamenta} it is shown that,  if $M$ is a compact Riemannian manifold of dimension $n\ge 2$, then for any continuous $\phi:M\to \R$ there exists a dense set of homeomorphisms of $M$ with a maximizing measure supported on aperiodic orbit, but in \cite{ATDCDS} it is shown that this set is meager. And in \cite{ATnonlin} it is shown that for a dense set of endomorphisms of the circle, there exists a $\phi$ maximizing measure supported on a periodic orbit. In this note we extend this last result, showing that

\begin{thm}\label{maintheorem}
Let $M$ be a compact Riemannian manifold and $\phi_0:M\to\R$ continuous. Then there exists a dense subset $\mathcal{A}$ of End($M$) such that, for every $f\in\mathcal{A}$ there exists a $\phi_0$ maximizing measure supported on a periodic orbit.
\end{thm}
Where End($M$) is the set of continuous surjections of $M$ endowed with the $\mathcal{C}^0$ metric, $d(f,g)=\sup_{x\in M}d(f(x),g(x))$.

The strategy of the proof, similar to the one used in \cite{ATnonlin}, is to make a series of local perturbations  in order to obtain a periodic source with large $\phi_0$ average while controlling the Birkhoff averages of the return map to the perturbation support. The proof of \cite{ATnonlin} relied on the local ordered structure of the domain, particularly in the definition of the support of the perturbations and in controlling the Birkhoff averages, two key points that were not adaptable to higher dimensions. In here we dealt with these difficulties by supporting perturbations in convex sets and analyzing the maximal Birkhoff sums on homothetic copies of the perturbation support, and by controlling the radial rate of escape from the periodic source. 

The paper is organized as follows: In the next section we present some preliminary lemmas and notations, and in section 3 prove the theorem. Since the argument is perturbative, for a given endomorphism we analyze several possibilities, each dealt with in a different subsection, and show for each possibility how to construct a
perturbed endomorphism close to $f$ with the desired property.

\section{Preliminaries}
We start with some notations and preliminary results. Let $M$ be a compact Riemannian manifold and End($M$) the set of endomorphisms of $M$, its continuous surjections. We endow End($M$) with its usual topology of uniform convergence and define the metric $d(f,g)= \sup_{x\in M}(d(f(x), g(x)),\, f,g \in \hbox{End(}M)$.

Given $f\in\hbox{End(}M)$ we denote by  $\mathcal M_{inv}(f)$ the set of $f$ invariant Borel probability measures, which is non-empty, convex  and also compact in the weak-* topology. The subset of ergodic measures of $\mathcal{M}_{inv}(f)$ is denoted by $\mathcal M_{erg}(f)$.

Given $\phi:M\to\R$ continuous and $f\in \hbox{End(}M)$, we define $P_\phi:\mathcal{M}_{inv}(f)\to\R, \, P_{\phi}(\mu)=\int \phi d\mu$. As the functional $P_\phi$ is affine and  $\mathcal M_{inv}(f)$ is a convex compact set, $P_\phi$ must have a maximum point at an extremal point of  $\mathcal M_{inv}(f)$. Since the extremal points of  $\mathcal M_{inv}$ are precisely the ergodic measures, there exists  some $\mu_{\max} \in \mathcal M_{erg}(f)$ that maximizes $P_\phi$. We denote $S_nf(x):=\displaystyle \sum_{i=0}^{n-1}\phi(f^i(x))$ to the $n$ Birkhoff sum of $\phi$.

The following lemma is a direct consequence of Atkinson's Lemma (see \cite{atkinson})
\begin{lemma}\label{atik}
Let $\phi:M\rightarrow \R$ be a continuous function, $f\in $End($M)$ and $\mu \in \mathcal M_{erg}(f)$, such that,
$$\int \phi(x)d\mu(x)=0.$$
Then for $\mu$-almost all $x \in M$, there exist $n_k\to\infty$ such that,
$$f^{n_k}(x)\rightarrow x \: \mbox{ and } \: S_{n_k}f (x) \rightarrow 0.$$ 
\end{lemma}

We begin with the following simple result

\begin{lemma}
There exists a dense subset of End($M$) such that, for all $x\in M$ and all $f$ in this subset, the set $ \{y\in M: f(y)=x \}$ is finite.
\end{lemma}

\begin{proof}

Let $f_0$ be an endomorphism, and let $\varepsilon>0$. We will find some $f$ with the stated property $\varepsilon$ close to $f_0$. First, let $\delta>0$ be such that, for all $x_1,x_2\in M$, if $d(x_1,x_2)<\delta$, then $d(f(x_1),f(x_2))<\varepsilon/2K$, where $K$ is the ratio of the radii of the circunscribed and inscribed spheres in the $n$ dimensional regular simplex.

Since every $n$-dimensional differential manifold admits a triangulation and $M$ is compact, we can assume that $M$ has a triangulation $\Im_1$ with finitely many triangles, such that each simplex has diameter less then $\delta$ and let $\Im_2$ be a subtriangulation of $\Im_1$ such that, for each $\Delta_i\in\Im_1$ there exists  some $\tilde\Delta_{j_i}\in\Im_2$ which is contained in the interior of $\Delta_i$.
 
Now we define $f:M\rightarrow M$, in a way that $f$ is a linear bijection in each triangle of $\Im_2$ and such that, in local coordinates, $f(\tilde\Delta_{j_i})$ is a simplex that contains $f_0(\Delta_i)$ and is contained in a sphere or radius $\varepsilon/2$. It should be immediate that $f$ is  a continuous surjection, since $M=\bigcup_{i\in I}f_0(\Delta_i)\subset \bigcup_{j\in J}f(\tilde\Delta_{j_i})\subset f(M)$. Moreover if $x\in \tilde\Delta_j$ for some $\Delta \in \Im$ then $d(f_0(x),f(x))<\varepsilon/2$, this implies $d(f_0,f)<\varepsilon.$

As $f$ is linear in each simplex of $\Im_2$, the set $\{y\in \tilde\Delta_j \cap M: f(y)=x \}$ is either empty or unitary, and therefore $\{y\in M: f_2(y)=x \}$ has cardinality smaller than or equal to the number of simplexes in $\Im_2$

\end{proof}

The structure of proof of Theorem \ref{maintheorem} is the following. Let $\phi_0:M\to\R$ be fixed. We start with an endomorphism $f$ which we assume that, for every $x\in M$, the pre-image of $x$ is finite and we construct successive small perturbations to produce an endomorphism $\widetilde f$ which is $\varepsilon$ close to $f$ and such that $\widetilde f$ has a $\phi_0$ maximizing measure supported on a periodic orbit.

Let $\mu_{\max}\in  \mathcal M_{erg}(f)$ be a $\phi_0$ maximizing measure and let $\phi= \phi_0-\int{\phi d\mu_{\max}}$, so that $\int{\phi d\mu_{\max}}=0$, and we remark that, for any endomorphism $g$, $\mu$ is a $\phi_0$ maximizing measure if and only if it is a $\phi$ maximizing measure.

\begin{lemma}\label{at}
For all $x\in \supp(\mu_{\max})$ and $\varepsilon>0$ there exist $y\in B_{\varepsilon}(x)$ and $n>0$ with $f^n(y)\in B_{\varepsilon}(x)$ and $S_nf(y)\ge0$. 
\end{lemma}

\begin{proof}

Let $x\in\supp(\mu_{\max})$. By \ref{atik}, since $\mu_{max}(B_{\varepsilon}(x))>0$, there exist $x_1\in \supp(\mu_{\max})$ and $n_k$, such that $f^{n_k}(x_1)\rightarrow x_1$ as $k\rightarrow \infty$, and such that $S_nf(x_1)\to 0$. Let $k_1>0$ be such that $f^{n_{k_1}}(x_1)\in B_{\varepsilon}(x)$ and let $a_1=S_{n_{k_1}}f(x_1)$. If $a_1\ge 0$ we set $y=x_1, n=n_{k_1}$ and we are done. If $a_1<0$, let $n_{k_2}>n_{k_1}$ be such that  $f^{n_{k_2}}(x_1)\in B_{\varepsilon}(x)$ and such that $S_{n_{k_2}}f(x_1)>a_1$. Then, as
$S_{n_{k_2}}f(x_1)= S_{n_{k_1}}f(x_1)+ S_{(n_{k_2}-n_{k_1})}f( f^{n_{k_1}}(x_1))$, we set $y = f^{n_{k_1}}(x_1)$ and $n = n_{k_2}-n_{k_1}$ and we are done.
\end{proof}

The next proposition is a consequence of the $\mathcal M_f$ compactness.

\begin{proposition}\label{jproposition}
For every constant $a>0$, there exist a positive integer $m_0=m_0(a)$ such that, for all $m\geq m_0$ and $x\in M$
$$\frac{1}{m}S_mf(x)\leq \frac{a}{2}. $$
\end{proposition}

\begin{proof}
This follows from $$\limsup_{n\rightarrow \infty} \max_{x\in M}\frac{1}{n}S_nf(x)=\sup_{\mu \in \mathcal M_{inv}(f)} \int{\phi d\mu}=0, $$ proposition 2.1 of  \cite{jenkinsondcds}
\end{proof}

%The title of your first subsection in section 2
\section{Contruction of the perturbed endomorphism}

Fix $x\in \supp(\mu)$ and let $\varepsilon>0$.  There are two possibilities,

\begin{itemize}
\item[I] {For all $y\in B_{\varepsilon}(x)$ and all $n>0$, if $f^n(y)\in B_{\varepsilon}(x)$ then $S_nf(y)\le 0$}
\item[II]{There exists $x_0\in B_{\varepsilon}(x)$ and $n_0>0$ such that $\widetilde f^{n_0}(x_0)\in B_{\varepsilon}(x)$ and $S_{n_0}f(x_0)>0$}.
\end{itemize}

\subsection{Case I}
Let us show first how to construct $\widetilde f$ in the case I :
Denote, for simplicity, $B=B_{\varepsilon}(x)$.
We assume that for all $y\in B$ and all $n>0$, if $f^n(y)\in B$ then $S_nf(y)\le 0$. From Lemma \ref{at} there exists $x_0\in B$ and $n_0>0$ such that $S_{n_0}f(x_0)\ge 0$, and so $S_{n_0}f(x_0)= 0$. Let $0<n_1\le n_0$ be the first return of $x_0$ to $B$. Note that, as 
$$0= S_{n_0}f(x_0)= S_{n_1}f(x_0)+S_{n_0-n_1}f(f^{n_1}(x_0))\le S_{n_1}f(x_0),$$ where the inequality comes from assuming that we are in case I, then $S_{n_1}f(x_0)\ge 0$ and, again from the assumption, $S_{n_1}f(x_0)=0$.

Let $T:M\to M$ be a homeomorphism such that $T(f^{n_1}(x_0))=x_0$, and such that $T$ is the identity outside of $B$, let $\widetilde f= T\circ f$. Note that $x_0$ is a $n_1$ periodic point for $\widetilde f$. Let $\mu_1$ be the measure uniformly distributed on the points of the $\widetilde f$ orbit of $x_0$.

\begin{lemma}\label{casoI}
$\mu_1$ is a $\phi$ maximizing measure for $\widetilde f$
\end{lemma}

\begin{proof}
Clearly $\int \phi d\mu_1= \frac{1}{n_1}S_{n_1}\widetilde f(x_0)= \frac{1}{n_1}S_{n_1} f(x_0)=0$. Furthermore, if $z\in M$ is such that there exists $n_z$ such that, if $n>n_z$ then $\widetilde f^{n}(z)\notin B$, then $\limsup_{n\to\infty}\frac{1}{n}S_n\widetilde f(z)= \limsup_{n\to\infty}\frac{1}{n}S_n\widetilde f(\widetilde f^{n_z}(z))= \limsup_{n\to\infty}\frac{1}{n}S_n f(\widetilde f^{n_z}(z))\le 0$ where the last inequality comes from the fact that $P_\phi(\mu)\le 0$ for all $\mu\in \mathcal{M}_{inv}(f)$.

On the other hand, if the return times of $z$ to $B$  are $0\le t_0<t_1<t_2.....$ with $t_k\to\infty$, then 
$$\frac{1}{t_k}S_{t_k} \widetilde f(z)= \frac{1}{k}\left(S_{t_0}\widetilde f(z) +\sum_{j=1}^{k} S_{t_j-t_{j-1}}\widetilde f( \widetilde f^{t_{j-1}}(z))\right)\le \frac{1}{k}S_{t_0}\widetilde f(z)\to 0$$ 
so that $\lim_{n\to\infty} \frac{1}{n} S_n\widetilde f(z)\le \int \phi d\mu_1$ for all $z\in M$ and we have the result
\end{proof}

\subsection{Case II }

 Assume now we are in case II, and let $a_0= \frac{1}{n_0}S_{n_0}f(x_0)>0$.

Denote by $B_{\varepsilon}[z]$ the closed ball with center $z$ and radius $\varepsilon$.

Let $m_0=m_0(a_0)>n_0>0$ be the integer from Proposition (\ref{jproposition}), and for each $k\in \{1,2,\ldots,m_0\}$ consider the compact sets 
$$K_k=B_{\varepsilon}[x_0]\cap f^{-k}(B_{\varepsilon}[x_0]).$$ 
For each $k$, let $c_k=\sup_{z\in K_k}\frac{1}{k}S_kf(z)$ and let $\overline{c}=\sup\{c_1, ... ,c_{m_0}\}$.  Note that, by the choice of $x_0,\, c_{n_0}\ge a_0.$ Furthermore, by Proposition (\ref{jproposition}), if $n>m_0$ then for all $z \in M,\,\frac{1}{n}S_nf(z)\le \frac{a_0}{2}.$ This, and the choice of $\overline c$ implies that, for each $z\in B_{\varepsilon}[x_0]$ and $n>0$ such that $f^n(z)\in B_{\varepsilon}[x_0]$, we have $\frac{1}{n}S_nf(z)\le \overline c$.

We consider 2 distinct possibilities:
\begin{itemize}
\item[(a)] {There  exists $q\in B_{\varepsilon}[x_0],\, n_q>0$ such that $\frac{1}{n_q}S_{n_q}f(q)=\overline c$ and $f^{n_q}(q)$ lies in the {\it open} ball $B_{\varepsilon}(x_0)$}
\item[(b)] {For all $z\in B_{\varepsilon}[x_0], n>0$ if $\frac{1}{n}S_nf(z)=\overline c$ and $f^n(z)\in B_{\varepsilon}[x_0]$, then $f^n(z)\in \partial B_{\varepsilon}[x_0]$}
\end{itemize}

\subsubsection{Case $(a)$}
If $(a)$ happens then we can define $\widetilde{f}=T\circ f$, where $T$ is the identity outside of $B_{\varepsilon}[x_0]$ and  $T$ is an endomorphism of $B_{\varepsilon}[x_0]$ satisfying $T(f^{n_q}(q))=q$.

The next lemma show us that the $\widetilde{f}$ invariant measure supported on the periodic orbit of $q$ is a $\phi$-maximizing measure.

\begin{lemma}\label{somamenor}
For any $z\in M, \liminf_{n\to\infty}\frac{1}{n}S_n\widetilde{f}(z)\le \frac{1}{n_q}S_{n_q}\widetilde{f}(q)=\overline c$
\end{lemma}

\begin{proof}

Let $z \in M$ and first assume that  $z$ is such that there exists some $\overline n$ such that  $\widetilde{f}^i(z)\notin B_{\varepsilon}[x_0]$ whenever $i\ge \overline{n}$, then 
$$\liminf_{n\to\infty}\frac{1}{n}S_n\widetilde{f}(z)= \liminf_{n\to\infty}\frac{1}{n}S_n\widetilde{f}( \widetilde f^{\overline{n}}(z))=\liminf_{n\to\infty}\frac{1}{n}S_nf( \widetilde f^{\overline{n}}(z))\le 0$$
where the second equality follows from the fact that $f(y)=\widetilde{f}(y)$ whenever $y\notin B_{\varepsilon}[x_0]$, and the inequality follows since the maximal $\phi$ average for $f$ is 0, and from $\sup_{z\in M}\limsup_{n\to\infty}\frac{1}{n}S_nf( \widetilde f^{\overline{n}}(z))\le \sup_{\mu\in\mathcal{M}_{inv}(f)} \int \phi d\mu$. As $\frac{1}{n_q}S_{n_q}f(q)>0$, we are done in this case.

Now assume that there exists an increasing sequence of times $N_k\to\infty, k\ge 1$ such that $\widetilde{f}^{i}(z)$ belongs to $B_{\varepsilon}[x_0]$ if and only if $i=N_k$ for some integer $k$. Then it holds that 
$$\frac{1}{N_{k+1}-N_k}\widetilde{f}(\widetilde{f}^{N_k}(z))=\frac{1}{N_{k+1}-N_k}f(\widetilde{f}^{N_k}(z))$$ and 

\begin{eqnarray}
\frac{1}{N_k} S_{N_k}\widetilde{f}(z)&=& \displaystyle \frac{1}{N_k} \left( \sum_{i=0}^{N_1-1}\phi \circ \widetilde{f}^i(z)+\ldots+\sum_{i=N_{k-1}}^{N_k-1}\phi \circ \widetilde{f}^i(z) \right) \nonumber \\
&=& \displaystyle \frac{1}{N_k}\left(\frac{N_1-N_0}{N_1-N_0}S_{N_1-N_0}\widetilde{f}(z)+\ldots+ \frac{N_k-N_{k-1}}{N_k-N_{k-1}}S_{N_k-N_{k-1}}\widetilde{f}\left( \widetilde{f}^{N_{k-1}}(z)\right) \right) \nonumber \\
&=& \displaystyle \frac{1}{N_k}\left(\sum_{i=1}^{k}(N_i-N_{i-1})\frac{1}{N_i-N_{i-1}}S_{N_i-N_{i-1}}f(\widetilde{f}^{N_{i-1}}(z))\right) \label{continha1} \\
&\leq & \frac{1}{N_k}\sum_{i=1}^{k}(N_i-N_{i-1})\overline c= \overline c. \nonumber
\end{eqnarray}
Where the inequality (\ref{continha1}) follows from $S_kf(z)=S_k\widetilde{f}(z)$, as $\widetilde{f}^i(z) \notin B_{\varepsilon}[x_0]$, $0\leq i \leq k-1$.

\end{proof}

The previous lemma shows that, if $z$ is a typical point of an $\widetilde{f}$ ergodic invariant measure $\mu$, then $\lim_{n}\frac{1}{n}S_n\widetilde{f}(z)=\int\phi d\mu\le \frac{1}{n_q}S_{n_q}\widetilde{f}(q)$ and we are done.

\subsubsection{Case $(b)$}

There exists some $z_1\in B_{\varepsilon}[x_0]$ and $n_{z_1}>0$ such that $f^{n_{z_1}}(z_1)\in\partial B_{\varepsilon}[x_0]$, and such that $\frac{1}{n_{z_1}}S_{n_{z_1}}f(z_1)=\overline c$. Let us call $q_0= f^{n_{z_1}}(z_1)$.
 Since each point in $M$ has finitely many preimages, the set  $P=(\bigcup_{i=1}^{m_0}f^{-i}(q_0))\cap B_{\varepsilon}[x_0]$ is finite, as is
$$\tilde P=\{ z\in P\mid \exists n_z>0 \hbox{ such that} f^{n_z}(z)=q_0 \hbox{ and } \frac{1}{n_z}S_{n_z}f(z)=\overline c\}.$$ 
Let $q\in \tilde P$ be a point which is closest to $q_0$ and let $n_q$ be such that $f^{n_q}(q)=q_0$ and $\frac{1}{n_q}S_{n_q}(q)=\overline c$.
Finally, let $E$ be some closed convex set contained in $B_{\varepsilon}[x_0]$ such that, if $d(z_1,z_2)\ge d(q,q_0)$ and $z_1,z_2\in E$, then $\{z_1,z_2\}=\{q,q_0\}$.

\begin{proposition}\label{cconexa}
There exist $\delta >0$ such that, if $z$ is not in the connected component of $$f^{-j}\left(B_{\delta}[q_0]\right)\bigcap \left(B_{\delta}[q_0]\cup E \right)$$ that contains $q$, then $$\frac{1}{j}S_jf(z)<\overline c$$ for all $j\in \N^*$.
\end{proposition}
\begin{proof}
By the choice of $E,\,\tilde P\cap E=\{q\}$ and so for any $z\not=q$ in $P\cap E$ and $n_z$ such that $f^{n_z}(z)=q_0,\, \frac{1}{n_z}S_{n_z}f(z)$ is strictly smaller than $\overline c$. Thus, by the continuity of $f$ and $\phi$, there exist $\delta_1(z)>0$ such that if $d(z,y)<\delta_1(z)$ we have $\frac{1}{i}S_if(y)<\frac{1}{n_q}S_{n_q}f(q)$. Moreover, for each $\delta_1(z)$ there exist $\delta_2(z)$ such that the connected component of $f^{-i}(B_{\delta_2(z)}(q_0))$ which contain $z$ is contained in $B_{\delta_1(z)}(z)$. Finally there exists $\delta_3>0$ such that, if $f^{-i}(B_{\delta_3}(q_0))$ intersects $E$ then there is some point of $P$ in this component. By taking $\delta=\min_{x\in P}\{\delta_2(x),\delta_3\}$ we are done.
 
\end{proof}

%\begin{lemma}
%The maximum of the partial averages with lenght $n_x$ for a point $x \in E$ is atained at $q$. That is, if $x\in E$ and $n_x$ is the time of his first return to $E$ then
%$$\frac{1}{n_x}S_{n_x}{f}(x) \leq \frac{1}{n_q}S_{n_q}{f}(q).$$
%\end{lemma}
%\begin{proof}
%Given $x\in E$ if its time of first return $n_x$ to $E$ is equal to the time of first return to $B_{\varepsilon}[x_0]$ then $$\frac{1}%{n_x}S_{n_x}{f}(x) \leq \frac{1}{n_q}S_{n_q}{f}(q).$$ 
%If $x\in E$ return to the set $B_{\varepsilon}[x_0]$ before $E$, we can write $n_x=n_1+\ldots+n_k$ where $$f^{\sum_{l=1}^{i}{n_l}}%(x)\in B_{\varepsilon}[x_0] \setminus E, 1\leq i\leq k-1.$$
%This way we got
%\begin{eqnarray}
%\frac{1}{n_x}S_{n_x}{f}(x) & \leq & \frac{1}{n_x}\sum_{i=1}^{k}(n_i-n_{i-1})\frac{1}{n_i-n_{i-1}}S_{n_i-n_{i-1}}{f}(f^{n_{i-1}}(x)) \nonumber \\ 
%& \leq & \max_{i\in \{1,\ldots,k \}}\frac{1}{n_i-n_{i-1}}S_{n_i-n_{i-1}}f(f^{n_{i-1}}(x)) \leq \frac{1}{n_q}S_{n_q}f^{n_q}(q) \nonumber.
%\end{eqnarray}
%\end{proof}

Denote the set $E\cup B_{\delta}[q_0]$ by $I$, we will construct a new endomorphism $\widetilde{f}=T\circ f$, where $T|_{M\setminus I}(z)=z$ , and such that there exist a $\widetilde{f}$-periodic point in $I$ whose average is strictly positive.

Let
$$D=I\cap \left( \bigcup_{i=1}^{\infty} f^{-i}(I) \right).$$
Over $D$ we define the following functions:
$$N_{ret}(x)=\inf\{j\in \N^*:f^j(x)\in I \}$$
$$f_2(x)=f^{N_{ret}(x)}(x)$$
$$ \psi(x)=\frac{1}{N_{ret}(x)} \sum_{i-0}^{N_{ret}(x)-1}\phi(f^i(x))$$

By the Proposition \ref{cconexa}, let $W_0$ be the connected component of $f_2^{-1}(B_{\delta}[q_0])$ which contains $q$, if $z\in D$ and $\psi(z)>\psi(q)$, then $z\in W_0$. Denote by $z_{\max}$ the point in $\bar{W_0}$ that maximizes $\psi(z)$. Choose $\alpha \in W_0$ sufficiently close to $z_{\max}$ such that the inequality
\begin{equation}
4 m_0 [\psi(z_{\max})-\psi(\alpha)]\leq |\psi(\alpha)-\psi(q)|
\label{alfa}
\end{equation}
is true, and such that $f_2(\alpha)\in int(B_{\delta}[q_0])$.

Now we consider $L$ to be the line segment joining $\alpha$ and $f_2(\alpha)$, $T_1:M\rightarrow M$ an homeomorphism mapping $f_2(\alpha)$ to $\alpha$, that is, $T_1(f_2(\alpha))=\alpha$ and such that $T_1$ is the identity outside $V(L)$, where 
$$V(L)=\{ z\in M: d(z,L)< \delta_3 \},$$ and $\delta_3>0$, chosen such that $V(L)$ is contained in the interior of $I$.

In figure (\ref{T2}), the shadow part is the neighborhood of the line segment $L$.

%Na figura (\ref{T2}) a regi\~ao sombreada denota a vizinhan\c{c}a do segmento $\overline{\alpha f_2(\alpha)}$.

\begin{figure}[h]
	\centering
		\includegraphics[scale=0.4]{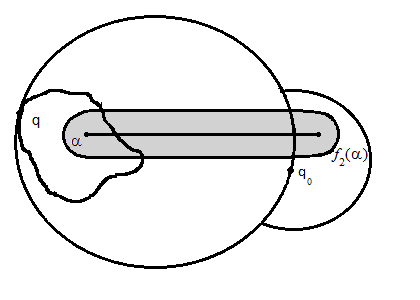}
	\caption{Neighborhood of $L$.}
	\label{T2}
\end{figure}

We define now $\widetilde f$ by the composition $\widetilde f = T_1\circ f$. Note that $\alpha$ is a $n_q$ periodic point for $\widetilde{f}$ and that the $\phi$ average over the orbit of $\alpha$ is $\psi(\alpha)\ge\frac{1}{n_q}S_{n_q}f(q)>0$. Yet the dynamics defined by $\widetilde{f}$ may have some new invariant measures whose $\phi$ average is strictly larger than $\psi(\alpha)$. 
Still, it should be clear that, as in the proof of Lemma \ref{somamenor} if $z$ is such that the $\widetilde{f}$ orbit of $z$ returns to $D$ finitely many times, then $\limsup_{n\to\infty}\frac{1}{n}S_nf(z)\le 0$, and if $z\in D$ returns infinitely-many times by $\widetilde{f}$ to the set $I$, but its orbit does not intersect $W_0$ (or just intersects it finitely many times), then if $n_1,n_2,\ldots$ are the return times to $D$, we have, by \ref{cconexa}:
 \begin{equation}\label{naoemw0}
\lim_{k\rightarrow \infty}\frac{1}{n_k}\sum_{i=0}^{n_k-1}\phi(\widetilde{f}^i(x))\leq \psi(q)=\overline c.
\end{equation}

Now we define $W_{\alpha}=W_0\bigcap f_2^{-1}\left( V(L) \right )$. If $z\in W_0 \setminus W_{\alpha}$ then $f_2(z)\in B_{\delta}[q_0]$ and we remark that $ B_{\delta}[q_0]$ is disjoint from $W_0$. So, if there is some future time $n_1>n_q$ such that $\widetilde{f}^{n_1}(z) \in W_0$,  we can write $n_1=n_q+k$ with $f^{n_q}(z)=f_2(z)\in B_{\delta}[q_0]$ and $f^{n_q+k}(z) \in W_{0}$. The following estimate will be useful

\begin{eqnarray}\label{naoemwalpha}
\displaystyle \frac{1}{n_{1}}\sum_{i=0}^{n_{1}-1}\phi(\widetilde{f}^i(z))-\psi(\alpha)=&\displaystyle \frac{n_q\psi(z)+k\psi(f_2(z))}{n_z+k}-\psi(\alpha) \nonumber \\
 \leq &\displaystyle \frac{n_q \psi(z_{\max})+k\psi(q)}{n_q+k}-\psi(\alpha) \nonumber \\ \leq & \displaystyle\frac{n_q \psi(z_{\max})+\psi(q)}{n_z+1}-\psi(\alpha) \nonumber \label{eqq}\\  \leq & \displaystyle\frac{n_q  (\psi(z_{\max})-\psi(\alpha))-(\psi(\alpha)-\psi(q))}{n_q+1}\leq 0, 
\end{eqnarray}
where the last inequality follows from (\ref{alfa}).

\subsection{The last pertubation}

In order to finish the demonstration of the Theorem (\ref{maintheorem}) we need to control the averages of those elements which have infinitely many returns on $W_{\alpha}$. In this section we construct a new pertubation $T_2$ such that $\alpha$ will be a  source for the new endomorphism $T_2 \circ \widetilde{f}$, and $W_{\alpha}$ is contained in its basin of repulsion.   

Let $\tilde{D}=I\cap \left( \bigcup_{i=1}^{\infty} \widetilde{f}^{-i}(I) \right )$ be the set of points who return to $I$ by the function $\widetilde{f}$. Over this set we define the following functions:
\begin{eqnarray}
\tilde{N}_{ret}(x)&=& \inf \{j\in \N^*: \widetilde{f}^j(x)\in I  \} \nonumber \\
\widetilde{f}_2(x)&=& \widetilde{f}^{\tilde{N}_{ret}(x)}(x) \nonumber \\
\tilde{\psi}(x)&=& \frac{1}{\tilde{N}_{ret}(x)}S_{\tilde{N}_{ret}(x)}\widetilde{f}(x) \nonumber
\end{eqnarray}

The following propositions are immediate from the definitions:
\begin{itemize}
\item[a)] $\tilde{D}=D$;
\item[b)] $\tilde{N}_{ret}(x)=N_{ret}(x)$ for all $x \in \tilde{D}$;
\item[c)] $\tilde{\psi}(x)=\psi(x)$ for all $x \in \tilde{D}$;
\item[d)] $\widetilde{f}_2(x)=T_1 \circ f_2(x)$ and if $x\notin f_2^{-1}\left(V(L)\right)$, then $\widetilde{f}_2(x)=f_2(x)$.
\end{itemize}

\

Let $\psi_{\max}:I\rightarrow \R$ be the following function:
$$\psi_{\max}(z)=\max_{y\in B_{d(\alpha,z)}[\alpha]}\psi(y).$$
Note that $\psi_{\max}(\alpha)=\psi(\alpha)$, $\psi(z)\leq \psi_{\max}(z)$, for all $z\in I$ and if $d(\alpha,z_1)<d(\alpha,z_2)$ then $\psi_{\max}(z_1)\leq \psi_{\max}(z_2).$

Define the function $P:\R \rightarrow \R$: 
$$P(s)=\sup_{z\in B_s(\alpha)} \psi(z)-\psi(\alpha). $$
This function is continuous non decreasing with $P(0)=0$, moreover, given $z$ with $d(z,\alpha)=s$, then $P(s)=\psi_{\max}(z)-\psi(\alpha)$.

Let $R_1, R_2,\, 0<R_1<R_2$ be such that $W_{\alpha}\subset B_{R_1}(\alpha) \subset B_{R_2}(\alpha)$ and $B_{R_2}(\alpha) \subset E$ as shown in figure (\ref{T1}): 

\begin{figure}[h]
	\centering
		\includegraphics[scale=0.7]{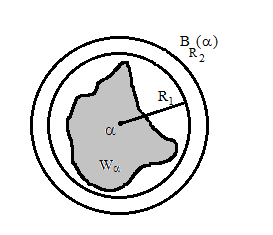}
	\caption{Perturbation region}
	\label{T1}
\end{figure}

The perturbation $T_2$ differs from the identity only at  $B_{R_2}(\alpha)$. Define the set $A_s=\{x\in B_{R_1}(\alpha):s\leq d(\alpha,z) \leq R_1 \}$ for all $0\leq s \leq R_1$ and we define the function
$$Q(s)=\inf_{z\in A_s}d(\widetilde{f}_{2}(z),\alpha).$$

$Q(s)$ is a non-decreasing continuous function. Since for all $z\in M,\,\widetilde{f}_2^{-1}(z)$ is a finite set,  $\alpha$ has only one pre-image by $\widetilde{f}_2$ in $B_{R_1}(\alpha)$. This implies $Q(0)=0$, and if $s>0$ then $Q(s)>0$.

Given $s_0=R_1$, we define two sequences $s_i$ and $r_i$ by:
$$P(s_i)=\frac{1}{2^i}P(R_1)\; \mbox{  and  }\; r_i=\min \{ Q(s_i), s_i \}.$$

Now we define the function $\lambda(z)$ by the rule:
If $r_{k+1}< d(z,\alpha)< r_{k}<R_1$ then $$s_k\leq \lambda(z) d(z,\alpha) \leq s_{k-1} .$$

One such function can be $\displaystyle \lambda(z)=\frac{s_k+(s_{k-1}-s_k)\frac{d(x,\alpha)-r_{k+1}}{r_k-r_{k+1}}}{d(z,\alpha)}$.

If $d(z,\alpha)>R_1$ then $\lambda(z)=1.$

Finally consider $\hat{f}=T_2\circ \widetilde{f}$, where $T_2$ on $B_{R_2}(\alpha)$ is the identity, and $T_2$ on $B_{R1}(\alpha)$ is defined by
$$T_2(z)=\lambda(z)(z-\alpha)+\alpha.$$ 

Note the function $T_2$ is a continuous function and a radial expansion with variable speed $\lambda(z)$.
Define $\hat{f}_2(z)=T_2\circ \widetilde{f}_2(z)$.

\begin{lemma} \label{lemaseq}
If $z\in W_{\alpha}$ with $d(\hat{f}_2(z),\alpha)<s_k$ then
$$d(z,\alpha) \leq s_{k+1}.$$
\end{lemma}
\begin{proof}
If $d(z,\alpha)>s_{k+1}$ then $d(\widetilde{f}_2(z), \alpha)>r_{k+1}$ this implies $$d(\hat{f}_2(z),\alpha)=d(\lambda(\widetilde{f}_2(z))(\widetilde{f}_2(z)-\alpha)+\alpha,\alpha)=\lambda(\widetilde{f}_2(z))d(\widetilde{f}_2(z),\alpha)\geq s_k ,$$ we can conclude, if $d(\hat{f}_2(z),\alpha)<s_k$, then $d(z,\alpha) \leq s_{k+1}.$
\end{proof} 

\begin{proposition}\label{proposicaox}
If $z, \hat{f}_2(z), \ldots, \hat{f}_2^N(z) \in W_{\alpha}$, then $d(z,\alpha) \leq s_N$.
\end{proposition}
\begin{proof}
Let us prove the proposition by induction over $N$. For $N=0$ we have $d(z,\alpha)\leq R_1$ therefore $d(z,\alpha) \leq s_0$, the induction hypothesis is that we assume true the assertion for $k=N-1$.

For $k=N$, if $y=\hat{f}_2(z)$ and as $y, \hat{f}_2(y), \ldots, \hat{f}_2^{N-1}(y) \in W_{\alpha}$ then by the induction hypothesis:
$$d(y,\alpha)\leq s_{N-1} \Rightarrow d(\hat{f}_2(z),\alpha)\leq s_{N-1} \mbox{ by lemma (\ref{lemaseq}) } d(z,\alpha) \leq s_N.$$
\end{proof}  

\begin{lemma} \label{lemadesig}
If $z, \hat{f}_2(z), \ldots, \hat{f}_2^k(z) \in W_{\alpha}$, then $N_{ret}(\hat{f}_2^i(z))=N_{ret}(q)$ for all \linebreak $i=1,\ldots,k$ and
$$\frac{1}{kN_{ret}(q)}S_{kN_{ret}(q)}\hat{f}(z)\leq \psi(\alpha)+\frac{1}{k}(\psi(x_{\max})-\psi(\alpha))$$
\end{lemma}

\begin{proof}
If $z, \hat{f}_2(z), \ldots, \hat{f}_2^k(z) \in W_{\alpha}$, then $N_{ret}(\hat{f}_2^i(z))=N_{ret}(q)$ for all $i=1,\ldots,k$.
Note that 
\begin{eqnarray}
\frac{1}{kN_{ret}(q)}S_{kN_{ret}(q)}\hat{f}(z)&=& \frac{1}{kN_{ret}(q)}\sum_{i=0}^{k-1}N_{ret}(q) \psi (\hat{f}_2^i(z)) \nonumber \\
& \leq & \frac{1}{k}\sum_{i=0}^{k-1}\psi_{\max} (\hat{f}_2^i(z)) \mbox{ and as } P(d(z,\alpha))=\psi_{\max}(z)-\psi(\alpha) \nonumber \\
& \leq & \psi(\alpha)+\frac{1}{k}\sum_{i=0}^{k-1}P(d(\hat{f}_2^i(z),\alpha)). \nonumber
\end{eqnarray}

By the proposition (\ref{proposicaox}) $d(\hat{f}_2^i(z),\alpha)\leq s_{k-i}$, therefore
$$P(d(\hat{f}_2^i(z),\alpha))\leq P(s_{k-i})=\frac{1}{2^{k-i}}P(R_1),$$
moreover, $P(R_1)\leq \psi(z_{max})-\psi(\alpha)$. This way we can conclude:

\begin{eqnarray}
\frac{1}{k} \sum_{i=0}^{k-1}P(d(\hat{f}_2^i(z),\alpha)) &\leq & \frac{1}{k}\left( \frac{1}{2^k}+\frac{1}{2^{k-1}}+\ldots +\frac{1}{2}\right) P(R_1) \nonumber \\
& \leq & \frac{1}{k}P(R_1) \leq \frac{1}{k} \left( \psi(x_{\max})-\psi(\alpha) \right), \nonumber
\end{eqnarray}
and we are done.
\end{proof}

As the $\phi$ integral over the measure equidistributed  over the $\hat f$ orbit of $\alpha$ is $\psi(\alpha)$, the final step in the proof of Theorem \ref{maintheorem} is
\begin{proposition}
For all $z\in M,\, \lim_{n\to\infty}\frac{1}{n}S_n\hat f(z)\le \psi(\alpha)$. 
\end{proposition}

\begin{proof}
First, note that, if the $\hat f$ orbit of $z$ does visit $W_0$ infinitely many times, then \ref{naoemw0} and the same argument applied in lemmas \ref{casoI} and \ref{somamenor} show the 
result.

Second, if the $\hat f$ orbit of $z$ visits $W_0$ infinitely many times, but only visits $W_{\alpha}$ a finite number of times, then using (\ref{naoemwalpha}) and again using the reasoning in lemmas \ref{casoI} and \ref{somamenor} we have the result.

Now assume $z\in W_{\alpha}$ returns to $W_{\alpha}$ infinitely many times, and let $n_j$ the sequence of times such that $\hat{f}^{n_j}(z)=(T_2\circ \widetilde{f})^{n_j}(x)\in I$, where $n_0=0$ and $n_{i+1}=n_i+N_{ret}(\hat f ^{n_i}(x))$. Consider the following subsequences of $(n_k)_{k\in \N}$:

\begin{itemize}

\item $a_j$, where $a_1=0$ and $a_{i+1}$ is the smallest integer larger than $a_i$ such that $\hat f ^{n_{a_{i+1}}}(x)$ is in  $W_{\alpha}$, but $\hat f ^{n_{a_{i+1}-1}}(x)$ is not.

\item $b_j$, where $b_i$ is the smallest integer larger than $a_i$ such that $\hat f ^{n_{b_i -1}}(x)$ is in $W_{\alpha}$, but $\hat f ^{n_{b_i}}(x)$ is not.

\end{itemize}

We have that:

\begin{eqnarray}\label{equacao12}
\frac{1}{n_{a_k}} S_{n_{a_k}}\hat f(x)=\displaystyle \frac{1}{n_{a_k}} \displaystyle \sum_{l=0}^{a_k-1}N_{ret}(\hat f^{n_l}(x))\psi(\hat f^{n_l}(x))=\\ = \frac{1}{n_{a_k}} \displaystyle \sum_{j=1}^{k-1} \left( \displaystyle \sum_{l=a_j}^{b_j-1} N_{ret}(\hat f^{n_l}(x))\psi(\hat f^{n_l}(x)) + \displaystyle \sum_{l=b_j}^{a_{j+1}-1} N_{ret}(\hat f^{n_l}(x))\psi(\hat f^{n_l}(x)) \right)\nonumber,
\end{eqnarray}

furthermore if $a_j\leq i\leq b_j-1$ then $\hat{f}^{n_i}(x)\in W_{\alpha}$ and if $b_j\leq i\leq a_{j+1}-1$ then $\hat{f}^{n_i}(x)\notin W_{\alpha}$, for all $x\in W_{\alpha}$, this way:

\begin{equation}
\displaystyle \sum_{l=b_j}^{a_{j+1}-1} N_{ret}(\hat f^{n_l}(x))\psi(\hat f^{n_l}(x))\leq (n_{a_{j+1}}-n_{b_j})\psi(q),\label{final1}
\end{equation}
For the other term in expression (\ref{equacao12}) we have:
%note que, $n_{a_{j+1}}=n_{b_j}+N_{ret}(\hat f^{n_{b_j}}(x))+\ldots+N_{ret}(\hat f^{n_{a_{j+1}-1}}(x)) \label{ajotas}$ e como $N_{ret}(x)=N_{ret}(q)$ para todo $x\in W_{\alpha}$ temos:

\begin{eqnarray}
& &\displaystyle \sum_{l=a_j}^{b_j-1} N_{ret}(\hat f^{n_l}(x))(\psi(\hat f^{n_l}(x))-\psi(\alpha))= \displaystyle \sum_{l=a_j}^{b_j-1} N_{ret}(q)\psi(\hat{f}^{n_l}(x))-(b_j-a_j)N_{ret}(q)\psi(\alpha) \nonumber \\
&\leq & (b_j-a_j) N_{ret}(q)\left( \psi(\alpha)+\frac{1}{b_j-a_j}[\psi(x_{\max})-\psi(\alpha) ] \right)-(b_j-a_j)N_{ret}(q)\psi(\alpha)\nonumber \\
&= & N_{ret}(q)(\psi(x_{\max})-\psi(\alpha))\leq \frac{\psi(\alpha)-\psi(q)}{2} \nonumber
\end{eqnarray}
which holds by lemma (\ref{lemadesig}) and by (\ref{alfa}) since $N_{ret}(q)< m_0$. Then
\begin{equation}
\displaystyle \sum_{l=a_j}^{b_j-1} N_{ret}(\hat f^{n_l}(x))\psi(\hat f^{n_l}(x)) \leq \frac{\psi(\alpha)-\psi(q)}{2}+(n_{b_j}-n_{a_j})\psi(\alpha). \label{final2}
\end{equation} Replacing (\ref{final1}) and (\ref{final2}) in (\ref{equacao12}) we have:

\begin{equation}
\frac{1}{n_{a_k}}\displaystyle \sum_{j=1}^{k-1}\left[ \psi(\alpha)\left(n_{b_j}-n_{a_j}-\frac{1}{2} \right)+\psi(q) \left( n_{a_{j+1}}-n_{b_j}+\frac{1}{2} \right) \right]
\end{equation}

by the manner $\alpha$ was chosen we have $\psi(\alpha)\geq \psi(q)$, so
\begin{eqnarray}
\frac{1}{n_{a_k}}\displaystyle \sum_{j=1}^{k-1}\left[ \psi(\alpha)\left(n_{b_j}-n_{a_j}-\frac{1}{2} \right)+\psi(q) \left( n_{a_{j+1}}-n_{b_j}+\frac{1}{2} \right) \right] \nonumber \\
\leq \frac{1}{n_{a_k}}\displaystyle \sum_{j=1}^{k-1} \psi(\alpha)(n_{a_{j+1}}-n_{a_j})=\psi(\alpha)-\frac{n_{a_1}}{n_{a_k}}\psi(\alpha)\rightarrow \psi(\alpha).\nonumber
\end{eqnarray}
concluding the proof of the proposition and the theorem

\end{proof}

\bibliographystyle{plain}
\bibliography{bibliografia}

\end{document}